\documentclass{amsart}
\usepackage{amsmath, amssymb, amsthm, amsfonts}
\usepackage[all]{xy}
\usepackage{graphicx}
\usepackage{tabularx}
\usepackage{hyperref}
\usepackage{tikz-cd}
\usepackage{cleveref}
\usepackage{doi}

\usepackage{enumerate}

\theoremstyle{plain}
\newtheorem{Thm}{Theorem}[section]
\newtheorem{Prop}[Thm]{Proposition}
\newtheorem{Cor}[Thm]{Corollary}

\newtheorem{Conj}[Thm]{Conjecture}

\theoremstyle{definition}
\newtheorem{Def}[Thm]{Definition}

\theoremstyle{remark}
\newtheorem{Rmk}[Thm]{Remark}
\newtheorem{Ex}[Thm]{Example}

\newcommand\Sing{\textup{Sing\,}}
\newcommand\codim{\textup{codim\,}}
\newcommand\reg{\textup{reg\,}}
\newcommand\depth{\textup{depth\,}}

\newcommand\pd{\textup{pd\,}}
\newcommand\Unpr{\textup{Unpr}}
\newcommand\Tor{\textup{Tor}}

\numberwithin{equation}{section}

\begin{document}

\title[Castelnuovo--Mumford regularity of unprojections]{Castelnuovo--Mumford regularity of unprojections and the Eisenbud--Goto regularity conjecture}
\author[Junho Choe]{Junho Choe}
\address{School of Mathematics, Korea Institute for Advanced Study (KIAS), 85 Hoegiro Dongdaemun-gu, Seoul 02455, Republic of Korea}
\email{junhochoe@kias.re.kr}

\date{\today}

\begin{abstract}
McCullough and Peeva found sequences of counterexamples to the Eisenbud--Goto conjecture on the Castelnuovo--Mumford regularity by using Rees-like algebras, where entries of each sequence have increasing dimensions and codimensions. In this paper we suggest another method to construct counterexamples to the conjecture with any fixed dimension $n\geq3$ and any fixed codimension $e\geq2$. Our strategy is an unprojection process and utilizes the possible complexity of homogeneous ideals with three generators. Furthermore, our counterexamples exhibit how singularities affect the Castelnuovo--Mumford regularity.
\end{abstract}

\maketitle
\setcounter{page}{1}

\section{Introduction}
Throughout this paper, we work over an algebraically closed field $\Bbbk$ of any characteristic, every \emph{variety} is always irreducible and reduced, and if its dimension is $n$, then it is called an \emph{$n$-fold} for brevity. Let $X\subseteq\mathbb P^r$ be a projective variety which is \emph{nondegenerate}, that is, it lies in no smaller subspaces. Write $S(X)=S(\mathbb P^r)/I(X)$ for the homogeneous coordinate ring and ideal of $X\subseteq\mathbb P^r$.

The \emph{Castelnuovo--Mumford regularity}, or simply the \emph{regularity}, is one of the most important invariants in projective algebraic geometry. It is used to express homological complexity as follows. Let $S=S(\mathbb P^r)$ be the polynomial ring in $r+1$ variables with the standard grading, and consider a finitely generated graded $S$-module $M$ with minimal (graded) free resolution
$$
\xymatrix{
0 & M \ar[l] & F_0 \ar[l] & F_1 \ar[l] & \cdots \ar[l] & F_i \ar[l] & \cdots, \ar[l] 
}
$$
where the graded $S$-free modules $F_i$ can be written as
$$
F_i=\bigoplus_{j\in\mathbb Z}S^{\beta_{i,j}(M)}(-i-j).
$$
Then the \emph{regularity} of $M$ is defined to be 
$$
\reg M=\max\{j\in\mathbb Z:\beta_{i,j}(M)\neq 0\text{ for some }i\},
$$
and the \emph{regularity} of $X\subseteq\mathbb P^r$, denoted by $\reg X$, is defined to be that of $I(X)$. 

Note that the regularity also measures cohomological complexity in relation of the Serre vanishing theorem. Refer to the well-known fact that for a coherent sheaf $F$ on $\mathbb P^r$ if its section module
$$
M=\bigoplus_{j\in\mathbb Z}H^0(\mathbb P^r,F(j))
$$
with degree pieces $M_j=H^0(\mathbb P^r,F(j))$ is finitely generated over $S$, then an integer $q$ satisfies
$$
H^i(\mathbb P^r,F(q-i))=0\quad\text{for every $i>0$}
$$
precisely when $q\geq\reg M$.

The \emph{Eisenbud--Goto regularity conjecture}, or the \emph{regularity conjecture} for short, says that the regularity of $X\subseteq\mathbb P^r$ has an upper bound in terms of the degree and codimension of $X$ as follows.

\begin{Conj}[Eisenbud--Goto \cite{eisenbud1984linear}]
Let $X\subseteq\mathbb P^r$ be any nondegenerate projective variety. Then we have
$$
\reg X\leq\deg X-\codim X+1.
$$
\end{Conj}
This conjecture holds true for
\begin{enumerate}
\item any projective variety $X\subseteq\mathbb P^r$ such that $S(X)$ is Cohen--Macaulay \cite{eisenbud1984linear},

\item any curves \cite{castelnuovo1893sui,gruson1983theorem},

\item smooth or mildly singular surfaces over $\mathbb C$ \cite{pinkham1986castelnuovo,lazarsfeld1987sharp,niu2015castelnuovo}, and

\item smooth or mildly singular threefolds in $\mathbb P^5$ over $\mathbb C$ \cite{kwak1999castelnuovo,niu2022castelnuovo}.
\end{enumerate}
For other related results see for example \cite{ran1990local,kwak2000generic,noma2014generic,kwak2020bound}.

But counterexamples to the regularity conjecture have been found by McCullough and Peeva \cite{mccullough2018counterexamples}. Their method uses both the \emph{Rees-like algebra} and the existence of homogeneous (nonprime) ideals having extremely big regularity. In fact, McCullough and Peeva showed that for any univariate (real) polynomial $p$, one can find a nondegenerate projective variety $X\subset\mathbb P^r$ of large dimension and codimension such that
$$
\reg X> p(\deg X).
$$
Moreover, they presented a computational example which is a nondegenerate singular projective threefold in $\mathbb P^5$ of degree $31$ and regularity $38$.

Our main result is an \emph{unprojection} process that is able to construct a collection of counterexamples to the regularity conjecture with any fixed dimension $n\geq3$ and any fixed codimension $e=r-n\geq2$, based on \emph{Green's partial elimination theory} \cite{green1998generic}. Our method works due to the possible complexity of homogeneous ideals with three generators.

Now let us state our main theorem. We need the following asymptotic notation. If $(a_d)$ and $(b_d)$ are sequences of positive numbers, then we write
$$
a_d=\Omega(b_d)\quad\text{when}\quad 0<\liminf_{d\to\infty}\frac{a_d}{b_d}\leq\infty.
$$

\begin{Thm}\label{ce}
Fix integers $n\geq 3$ and $e\geq2$. Then there exists a sequence 
$$
(X_d\subset\mathbb P^{n+e}:d=2,3,\ldots)
$$ 
of nondegenerate projective $n$-folds in $\mathbb P^{n+e}$ such that ${\displaystyle \lim_{d\to\infty}\deg X_d=\infty}$, and
$$
\reg X_d=\Omega((\deg X_d)^k)\quad\text{with}\quad k=\left\lfloor\frac{n+1}{2}\right\rfloor.
$$
\end{Thm}

This paper is organized as follows. In \Cref{secunpr} we introduce the partial elimination theory and the unprojection. In \Cref{secinvariant} we give preliminary results to compute significant invariants such as degree, regularity, and projective dimension. In \Cref{secfamily} we list certain families of homogeneous ideals whose regularity values (seem to) rapidly increase. Some of them become ingredients for the proof of the main theorem. In \Cref{secproof} we finally prove the main theorem, and a specific counterexample to the regularity conjecture is given in detail.

The following assumptions and notations are used.
\begin{itemize}
\item $X\subseteq\mathbb P^r$ and $Y\subseteq\mathbb P^{r-1}$ are nondegenerate projective varieties.

\item $S(X)=S(\mathbb P^r)/I(X)$ and $S(Y)=S(\mathbb P^{r-1})/I(Y)$ are homogeneous coordinate rings and homogeneous ideals of $X\subseteq\mathbb P^r$ and $Y\subseteq\mathbb P^{r-1}$, respectively.

\item We simply write $S=S(\mathbb P^r)$ and $R=S(\mathbb P^{r-1})$.

\item $V(I)$ is the projective scheme defined by a homogeneous ideal $I$.

\item For a graded module or ring $M$ we denote by $M_j$ its piece in degree $j$.

\item The missing entries in matrices and tables are usually assumed to be zero.
\end{itemize}

\bigskip

\section{Projection and unprojection}\label{secunpr}

\emph{Unprojection} is the reverse process of projection. In this section we outline projections from a point and associated unprojections from the perspective of \emph{Green's partial elimination theory} \cite{green1998generic}. For another approach to unprojections consult with \cite{reid2000graded}.

\bigskip

\subsection{Partial elimination theory}

Fix a point $z\in\mathbb P^r$ inside or outside $X\subseteq\mathbb P^r$, and consider the projection $\pi_z:\mathbb P^r\setminus z\to\mathbb P^{r-1}$ from $z$. In order to analyze geometry of the restricted map $\pi_z|_X:X\setminus z\to\mathbb P^{r-1}$ Green introduced \emph{partial elimination ideals} \cite{green1998generic}. We collect some information on them below and refer to \cite[Appendix]{choe2022matryoshka} for their applications.

Taking suitable homogeneous coordinates $x_0,x_1,\ldots,x_r\in S_1$ in $\mathbb P^r$ we may assume that the point $z\in\mathbb P^r$ is given as
$$
(x_0:x_1:\cdots:x_r)=(1:0:\cdots:0).
$$
One thus has $\pi_z(x_0:x_1:\cdots:x_r)=(x_1:\cdots:x_r)\in\mathbb P^{r-1}$. Set $R=\Bbbk[x_1,\ldots,x_r]\subset S$ so that it forms the homogeneous coordinate ring of $\mathbb P^{r-1}$ and satisfies $R[x_0]=S$. For an integer $i$ and the graded $R$-module
$$
\widetilde{K}_i(X,z)=\{f\in I(X):f\text{ has degree at most }i\text{ in }x_0\}
$$ 
the \emph{$i$-th partial elimination ideal} of the pair $(X,z)$ is defined to be
$$
K_i(X,z)=\{\text{coefficients of $x_0^i$ in any }f\in\widetilde{K}_i(X,z)\}\subseteq R.
$$
Note that $K_0(X,z)$ is nothing but the elimination ideal $I(X)\cap R$ of $I(X)\subseteq S$. Partial elimination ideals are homogeneous ideals that fit into both the short exact sequence
\begin{equation}\label{graded}
\xymatrix{
0 \ar[r] & \widetilde{K}_{i-1}(X,z) \ar[r] & \widetilde{K}_i(X,z) \ar[r] & K_i(X,z)(-i) \ar[r] & 0
}
\end{equation}
and the chain
\begin{equation}\label{chain}
K_0(X,z)\subseteq K_1(X,z)\subseteq\cdots\subseteq K_i(X,z)\subseteq\cdots.
\end{equation}

We write $X_z=V(K_0(X,z))\subseteq\mathbb P^{r-1}$. Then since $K_0(X,z)$ is prime without linear forms, the projective scheme $X_z\subseteq\mathbb P^{r-1}$ is a nondegenerate projective variety with $I(X_z)=K_0(X,z)$, and one can see that $X_z$ has support 
$$
X_z=\overline{\pi_z(X\setminus z)}\subseteq\mathbb P^{r-1}.
$$
The induced map 
$$
\pi_z|_X:X\setminus z\to X_z
$$ 
is called the \emph{projection} of $X$ from $z$.

On the other hand, we consider the homogeneous ideal 
$$
K_\infty(X,z)=\bigcup_{i=0}^\infty K_i(X,z)\subseteq R
$$
to which the chain \eqref{chain} stabilizes. If $z\in X$, then $\mathbb PC_zX=V(K_\infty(X,z))\subseteq\mathbb P^{r-1}$ is called the \emph{projectivized tangent cone} to $X$ at $z$, and if we forget its embedding in $\mathbb P^{r-1}$, then it describes intrinsic and local geometry of $X$ at $z$. For example, one sees that $\mathbb PC_zX=\textup{Proj}\,\bigoplus_{j=0}^\infty m_z^j/m_z^{j+1}$ for the maximal ideal $m_z$ of the local ring of $X$ at $z$, and if $z$ is a smooth point of $X$, then $\mathbb PC_zX$ is just the projectivized tangent space to $X$ at $z$. Note that $\mathbb PC_zX$ is known to be equidimensional with $\dim\mathbb PC_zX=\dim X-1$. However, if $z\not\in X$, then we have $K_\infty(X,z)=R$.

Partial elimination ideals have the following set-theoretic description.

\begin{Prop}[{\cite[Proposition 6.2]{green1998generic} and \cite[Proposition A.2]{choe2022matryoshka}}]\label{setdescription}
Let $Z_i=V(K_i(X,z))$ and $Z_\infty=V(K_\infty(X,z))$ in $\mathbb P^{r-1}$. Then the support of $Z_i$ is given as
$$
Z_i=\{p\in\mathbb P^{r-1}:\textup{length}(X\cap\pi_z^{-1}(p))>i\}\cup Z_\infty\subseteq\mathbb P^{r-1},
$$
where $\pi_z^{-1}(p)\cong\mathbb A^1$ is the fiber of $p\in\mathbb P^{r-1}$ under $\pi_z:\mathbb P^r\setminus z\to\mathbb P^{r-1}$. In particular if $z$ is not a vertex of $X$, then $\pi_z|_X:X\setminus z\to X_z$ is generically finite of degree 
$$
\deg\pi_z|_X=\min\{i\geq1:Z_i\neq X_z\}.
$$
\end{Prop}
This is related to the well-known degree formula
\begin{equation}\label{degformula}
\deg X=\deg\pi_z|_X\deg X_z+
\begin{cases}
\deg\mathbb PC_zX&\text{if }z\in X \\
0 &\text{if }z\not\in X
\end{cases}
\end{equation}
for the case where $\pi_z|_X:X\setminus z\to X_z$ is generically finite.

We specify the easiest but nontrivial case in this regard. First suppose that $\pi_z|_X:X\setminus z\to X_z$ is birational so that $K_0(X,z)\neq K_1(X,z)$. However the behavior of the partial elimination ideals $K_1(X,z)\subseteq K_2(X,z)\subseteq\cdots\subseteq K_\infty(X,z)$ is still mysterious in general. So it would be desirable that they are the same. It is worth noting that this actually happens for interesting classes of nondegenerate projective varieties. See \cite{han2015sharp} and \cite{choe2022matryoshka}. 

\begin{Def}
One says that $\pi_z|_X$ is \emph{simple} if it is birational together with $K_1(X,z)=K_\infty(X,z)$.
\end{Def}

Notice that if $\pi_z|_X$ is simple, then $z\in\mathbb P^r$ lies in $X\subset\mathbb P^r$.

\bigskip

\subsection{Unprojection}

Now let us introduce the notion of unprojection. It depends on a given \emph{fake linear form} that plays the role of the so-called \emph{unprojection variable}.

Let $Y\subseteq\mathbb P^{r-1}$ be a nondegenerate projective variety, and consider the fraction field of $S(Y)$ that is equal to the function field $\Bbbk(\widehat{Y})$ of the affine cone $\widehat{Y}\subseteq\mathbb A^r$ of $Y\subseteq\mathbb P^{r-1}$. Now we take account of an element $\lambda$ in the algebraic closure of $\Bbbk(\widehat{Y})$, and assume that the minimal polynomial of $\lambda$ over $\Bbbk(\widehat{Y})$ has the form
$$
\lambda^d+\frac{a_1}{f_1}\lambda^{d-1}+\cdots+\frac{a_d}{f_d},
$$
where $d\geq 1$ is the degree of $\lambda$ over $\Bbbk(\widehat{Y})$, and the $a_i\in S(Y)$ and $f_i\in S(Y)\setminus0$ are homogeneous with $\deg a_i=\deg f_i+i$ for all $i=1,\ldots,d$.

\begin{Def}
Let $Y$ and $\lambda$ be as above. If $\lambda\not\in S(Y)_1$, then one says that $\lambda$ is a \emph{fake linear form} on $Y\subseteq\mathbb P^{r-1}$. Two fake linear forms on $Y\subseteq\mathbb P^{r-1}$ are called \emph{conjugate} if so are they as algebraic elements over $\Bbbk(\widehat{Y})$.
\end{Def}

Fake linear forms induce unprojections in the following sense.

\begin{Thm}[Unprojections]\label{corres}
Fix homogeneous coordinates $x_0,x_1,\ldots,x_r$ on $\mathbb P^r$ so that $\mathbb P^{r-1}$ has homogeneous coordinates $x_1,\ldots,x_r$ under the projection $\pi_z:\mathbb P^r\setminus z\to\mathbb P^{r-1}$ from the point $z=(1:0:\cdots:0)\in\mathbb P^r$. Then for any nondegenerate projective variety $Y\subseteq\mathbb P^{r-1}$ and any integer $d\geq 1$ there is a one-to-one correspondence
\begin{multline*}
\{X\subset\mathbb P^r:X_z=Y\subseteq\mathbb P^{r-1}\textup{, and }\pi_z|_X:X\setminus z\to Y\textup{ has degree }d\} \\
\longleftrightarrow\{\textup{fake linear forms }\lambda\textup{ of degree }d\textup{ on }Y\subseteq\mathbb P^{r-1}\textup{ up to conjugacy}\},
\end{multline*}
where $X\subset\mathbb P^r$ stands for a nondegenerate projective variety. In addition under this correspondence if $\pi_z|_X:X\setminus z\to Y$ is birational, that is, $d=1$, then we have 
\begin{equation}\label{colon}
\frac{K_i(X,z)}{K_0(X,z)}=(f(a,f)^{i-1}:a^i)\subseteq S(Y)    
\end{equation}
for each $i\geq 1$ and any fractional expression $\lambda=a/f\in\Bbbk(\widehat{Y})$.
\end{Thm}

\begin{proof}
The correspondence maps are sketched as follows.
$$
\begin{array}{rcl}
X\subset\mathbb P^r & \longmapsto & x_0\in\Bbbk(\widehat{X}) \\
\textup{Proj}\,S(Y)[\lambda]\subset\mathbb P^r & \reflectbox{$\longmapsto$} & \lambda
\end{array}
$$
They are obviously inverses of each other.

Let $X\subset\mathbb P^r$ be a nondegenerate projective variety being considered. We claim that the element $x_0\in\Bbbk(\widehat{X})$ is a fake linear form of degree $d$ on $Y\subseteq\mathbb P^{r-1}$. Indeed since $\pi_z|X:X\setminus z\to Y$ is generically finite of degree $d$, the function field $\Bbbk(\widehat{X})$ of the affine cone $\widehat{X}\subset\mathbb A^{r+1}$ is a degree $d$ extension of $\Bbbk(\widehat{Y})$. Thus, $x_0\in\Bbbk(\widehat{X})$ has degree $d$ over $\Bbbk(\widehat{Y})$, for $\Bbbk(\widehat{X})=\Bbbk(\widehat{Y})(x_0)$. Note that $I(Y)\neq K_d(X,z)$ by \Cref{setdescription}. Pick a homogeneous form $f\in K_d(X,z)\setminus I(Y)$ so that $fx_0^d+a_1x_0^{d-1}+\cdots+a_d\in I(X)$ for some homogeneous forms $a_i\in R$ with $\deg a_i=\deg f+i$. Hence,
$$
x_0^d+\frac{a_1}{f}x_0^{d-1}+\cdots+\frac{a_d}{f}
$$
is the minimal polynomial of $x_0\in\Bbbk(\widehat{X})$ over $\Bbbk(\widehat{Y})$. Also since $X\subset\mathbb P^r$ is nondegenerate, one sees that $x_0\not\in S(Y)_1$.

On the other hand, let $\lambda$ be a fake linear form under discussion. Then consider a graded $R$-algebra $S(Y)[\lambda]$ in the algebraic closure of $\Bbbk(\widehat{Y})$, where $\lambda$ is regarded as a homogeneous element of degree $1$ in it, and define a graded $R$-algebra epimorphism
$$
\phi:S\to S(Y)[\lambda]\quad\text{by}\quad\phi(x_0)=\lambda.
$$
Notice that $\ker\phi\subset S$ is a homogeneous prime ideal and that it has no linear forms since $\lambda\not\in S(Y)_1$. Take $X\subset\mathbb P^r$ to be the nondegenerate projective variety defined by $\ker\phi$. Evidently we have $K_0(X,z)=I(Y)$, that is, $X_z=Y$. Also $\pi_z|_X:X\setminus z\to Y$ is generically finite of degree $d$, for $\Bbbk(\widehat{X})=\Bbbk(\widehat{Y})(\lambda)$ is a degree $d$ extension of $\Bbbk(\widehat{Y})$.

For the last assertion let $g\in K_i(X,z)/K_0(X,z)$ be an element so that $g\lambda^i+b_1\lambda^{i-1}+\cdots+b_i=0$ in $\Bbbk(\widehat{Y})[\lambda]$ for some homogeneous forms $b_i\in S(Y)$ with $\deg b_i=\deg g+i$. Substituting $a/f$ for $\lambda$ and clearing denominators we have $ga^i=-b_1a^{i-1}f\cdots-b_if^i\in f(a,f)^{i-1}$, hence $g\in(f(a,f)^{i-1}:a^i)$. So $K_i(X,z)/K_0(X,z)\subseteq(f(a,f)^{i-1}:a^i)$. Similarly, the reverse containment holds.
\end{proof}

For convenience we use the following notation. It is well defined up to projective equivalence.
\begin{Def}
After choosing a point $z\in\mathbb P^r$ if \Cref{corres} associates a nondegenerate projective variety $X\subset\mathbb P^r$ to a fake linear form $\lambda$ on $Y\subseteq\mathbb P^{r-1}$, then we write
$$
(X,z)=\Unpr(Y,\lambda)
$$
and call it the \emph{unprojection} of the pair $(Y,\lambda)$.
\end{Def}

On the other hand, to compute invariants of the unprojections we should usually determine when the partial elimination ideals \emph{stabilize}, that is, for which values of $i$ the equality $K_i(X,z)=K_\infty(X,z)$ holds. The following helps us to recognize whether an unprojection $(X,z)$ has $\pi_z|_X$ simple.

\begin{Prop}\label{degsimple}
Let $\lambda\in\Bbbk(\widehat{Y})$ be a fake linear form of degree $1$ on $Y\subseteq\mathbb P^{r-1}$ with unprojection $(X,z)=\Unpr(Y,\lambda)$, and fix a fractional expression $\lambda=a/f\in\Bbbk(\widehat{Y})$ with $a,f$ homogeneous. Take $Z_i=V(K_i(X,z))\subseteq\mathbb P^{r-1}$ and $W_i=V((a,f)^i)\subsetneq Y\subseteq\mathbb P^{r-1}$ for each $i\geq 1$.  With $n=\dim Y$ we set
$$
\chi_{n-1}(W)=
\begin{cases}
\deg W & \textup{if }\dim W=n-1 \\
0 & \textup{if }\dim W<n-1
\end{cases}
$$
for any projective scheme $W\subsetneq Y\subseteq\mathbb P^{r-1}$, and consider the sequence
$$
(d_0,d_1,d_2,\ldots)=(0,\chi_{n-1}(W_1),\chi_{n-1}(W_2),\ldots).
$$
Then this sequence is convex and eventually linear with slope $\deg f\deg Y-\chi_{n-1}(Z_\infty)$, where $Z_\infty=V(K_\infty(X,z))\subseteq\mathbb P^{r-1}$. Furthermore, if $S(Y)$ is Cohen--Macaulay, then the following are equivalent.
\begin{enumerate}
    \item\label{degsimple1} $\pi_z|_X$ is simple.
    \item\label{degsimple2} $Z_1$ and $Z_\infty$ have the same dimension and degree.
    \item\label{degsimple3} $(d_i)$ is linear.
    \item\label{degsimple4} $d_i\leq d_1\cdot i$ for all $i\geq2$.
    \item\label{degsimple5} $z\in X$, and $\pi(X)=\pi(Y)+\pi(\mathbb C_zX)+\deg\mathbb C_zX-1$, where $\pi$ means the sectional genus, and $\mathbb C_zX$ is the cone over $\mathbb PC_zX$ with vertex $z$.
\end{enumerate}
\end{Prop}

\begin{proof}
The formula \eqref{colon} induces the following diagram.
$$
\xymatrix{
0 \ar[r]	& \dfrac{R}{K_i(X,z)}(-i\deg a) \ar[r]^--{a^i} 			& \dfrac{S(Y)}{f(a,f)^{i-1}} \ar[r] \ar@{=}[d]	& \dfrac{S(Y)}{(a,f)^i} \ar[r]	& 0 \\
0 \ar[r]	& \dfrac{S(Y)}{(a,f)^{i-1}}(-\deg f) \ar[r]^--f				& \dfrac{S(Y)}{f(a,f)^{i-1}} \ar[r] 					& \dfrac{S(Y)}{(f)} \ar[r] 		& 0 
}
$$
From Hilbert polynomials of the graded $R$-modules in this diagram it follows that
\begin{equation}\label{difference}
\chi_{n-1}(Z_i)=\deg f\deg Y-(d_i-d_{i-1})
\end{equation}
for every integer $i\geq 1$. Refer to the fact that
$$
\chi_{n-1}(Z_1)\geq\chi_{n-1}(Z_2)\geq\cdots\geq\chi_{n-1}(Z_\infty)\geq0.
$$
Therefore, what remains to check is the last assertion. 

We prove the equivalence of \eqref{degsimple1} and \eqref{degsimple2}. Let $\overline{K_1}$ and $\overline{K_\infty}$ be the images of $K_1(X,z)$ and $K_\infty(X,z)$, respectively, in $S(Y)$. Since $\overline{K_1}=((f):a)$ by \Cref{corres}, it is unmixed in the Cohen--Macaulay ring $S(Y)$. Consider the short exact sequence
$$
\xymatrix{
0 \ar[r] & \overline{K_\infty}/\overline{K_1} \ar[r] & S(Y)/\overline{K_1} \ar[r] & S(Y)/\overline{K_\infty} \ar[r] & 0
}
$$
and the containment 
$$
\textup{Ass}_{S(Y)}(\overline{K_\infty}/\overline{K_1})\subseteq\textup{Ass}_{S(Y)}(S(Y)/\overline{K_1})
$$ 
of sets of associated primes. One finds that $\overline{K_\infty}/\overline{K_1}=0$ if and only if $\chi_{n-1}(Z_1)=\chi_{n-1}(Z_\infty)$. Also, the statements \eqref{degsimple2}, \eqref{degsimple3}, and \eqref{degsimple4} are equivalent because of \eqref{difference} and the convexity of $(d_i)$. For the equivalence of \eqref{degsimple5} and the others refer to a formula \cite[Theorem A.4(2)]{choe2022matryoshka}.
\end{proof}

\bigskip

\section{Regularity and other invariants}\label{secinvariant}

This section is devoted to computing invariants by using the partial elimination theory. Let $M$ be a finitely generated graded $S$-module with minimal free resolution
$$
\xymatrix{
F_0 & F_1 \ar[l] & \cdots \ar[l] & F_i \ar[l] \ar@{=}[d] & \cdots. \ar[l] \\
&&& {\displaystyle\bigoplus_{j\in\mathbb Z}}S^{\beta_{i,j}(M)}(-i-j) &
}
$$
The exponents $\beta_{i,j}(M)$ are called \emph{graded Betti numbers} of $M$ and encode basic information of generators, relations, and \emph{syzygies} of $M$. Notice that by nature of Tor functors we have
$$
\beta_{i,j}(M)=\dim_\Bbbk\Tor_i^S(M,\Bbbk)_{i+j},
$$
where $\Bbbk$ denotes the residue field of $S$, and $\Tor_i^S(M,\Bbbk)$ can be computed as the (co)homology group of the Koszul type complex
$$
\xymatrix{
\bigwedge^{i+1}S_1\otimes M(-i-1) \ar[r] & \bigwedge^iS_1\otimes M(-i) \ar[r]^--d & \bigwedge^{i-1}S_1\otimes M(-i+1)
}
$$
at the middle, where the differential $d$ is given by combining the natural injection
$$
\bigwedge^iS_1\hookrightarrow\bigwedge^{i-1}S_1\otimes S_1
$$
and the multiplication map
$$
S_1\otimes M(-i)\to M(-i+1).
$$

\begin{Def}
Let $z\in\mathbb P^r$ be a point. We define the \emph{$i$-th partial elimination module} of $(X,z)$ to be
$$
M_i(X,z)=\frac{K_i(X,z)}{K_{i-1}(X,z)}(-i).
$$
\end{Def}

Syzygies of the $M_i(X,z)$ are building blocks of those of $I(X)$.

\begin{Thm}[cf.\ {\cite[Proposition 6.6]{green1998generic}}]\label{spectral}
Let $z\in\mathbb P^r$ be a point. Then there is a spectral sequence $E$ of graded $R$-modules such that
$$
E^1_{p,q}=\Tor^R_{p+q}(M_p(X,z),\Bbbk)\Rightarrow \Tor^S_{p+q}(I(X),\Bbbk).
$$
\end{Thm}

\begin{proof}
Pick homogeneous coordinates $x_0,x_1,\ldots,x_r$ in $\mathbb P^r$ such that $z=(1:0:\cdots:0)\in\mathbb P^r$.
By a well-known Lefschetz type theorem we have
$$
\Tor^S_\ast(I(X),\Bbbk)\cong\Tor^R_\ast(I(X)/x_0I(X),\Bbbk).
$$
Thus, it suffices to take account of the graded $R$-module $I(X)/x_0I(X)$. 

Now let us see the following Koszul type complex and its ascending filtration:
$$
C_i=\bigwedge^iR_1\otimes\frac{I(X)}{x_0I(X)}(-i)\quad\text{and}\quad F_pC_i=\bigwedge^iR_1\otimes\frac{\widetilde{K}_p(X,z)}{x_0\widetilde{K}_{p-1}(X,z)}(-i).
$$ 
This filtration is given by the short exact sequence
$$
\xymatrix{
0 \ar[r] & \dfrac{\widetilde{K}_{p-1}(X,z)}{x_0\widetilde{K}_{p-2}(X,z)} \ar[r] & \dfrac{\widetilde{K}_p(X,z)}{x_0\widetilde{K}_{p-1}(X,z)} \ar[r] & M_p(X,z) \ar[r] & 0
}
$$
that comes from the short exact sequence (\ref{graded}). Consider the spectral sequence $E$ induced by the filtered complex $C_\ast$. One has
\begin{align*}
E^0_{p,q} & = G_pC_{p+q}=\bigwedge^{p+q}R_1\otimes M_p(X,z)(-p-q), \\
E^1_{p,q} & = H_q(G_pC_{p+\ast})=\Tor^R_{p+q}(M_p(X,z),\Bbbk),\text{ and} \\
E^\infty_{p,q} & = G_pH_{p+q}(C_\ast)=G_p\Tor^R_{p+q}(I(X)/x_0I(X),\Bbbk),
\end{align*}
where $G$ means the associated graded object. We have constructed the desired spectral sequence.
\end{proof}

\begin{Rmk}
Take integers $i$ and $k$, and consider the spectral sequence $E$ above. Then the differential
$$
E^1_{k,i-k}=\Tor_i^S(M_k(X,z),\Bbbk)\to E^1_{k-1,i-k}=\Tor_{i-1}^S(M_{k-1}(X,z),\Bbbk)
$$
is derived from a short exact sequence of the following form.
$$
\xymatrix{
0 \ar[r] & M_{k-1}(X,z) \ar[r] & \dfrac{\widetilde{K}_k(X,z)}{x_0\widetilde{K}_{k-1}(X,z)+\widetilde{K}_{k-2}(X,z)} \ar[r] & M_k(X,z) \ar[r] & 0
}
$$
This is connected to the \emph{Griffiths--Harris second fundamental form} for $X$ at $z$ \cite{griffiths1979algebraic}. Refer to \cite[Definition 6.12 and Example 6.13]{green1998generic}.
\end{Rmk}

We suggest a simple trick for partial computation of the spectral sequence above. We first adopt some terminology.
\begin{Def}
A \emph{table} is a $\mathbb Z\times\mathbb Z$ matrix of nonnegative integers. For instance the \emph{Betti table} $\mathbb B(M)$ of a finitely generated graded $S$-module $M$ is a table defined by
$$
\mathbb B(M)_{i,j}=\beta_{i,j}(M).
$$
A \emph{table sequence} is a collection of tables indexed by $\mathbb Z$. For a table sequence $\mathcal B=(B_k:k\in\mathbb Z)$ its \emph{cancellation} is any table sequence that can be obtained from $\mathcal B$ by repeating the following operation a finite number of times. For a tuple $(i,j,k,l)$ of integers with $l>0$ subtract a common nonnegative integer from both
$$
(B_k)_{i,j}\quad\text{and}\quad(B_{k-l})_{i-1,j+1}.
$$
For a table $B$ its \emph{decomposition} is a table sequence $(B_k:k\in\mathbb Z)$ such that for each $(i,j)\in\mathbb Z\times\mathbb Z$ the entries $(B_k)_{i,j}$ sum up to $B_{i,j}$.
\end{Def}

Then the convergence of the spectral sequence above can be roughly described as follows.

\begin{Prop}[Cancellation principle]\label{cancel}
For any fixed point $z\in\mathbb P^r$ the Betti table $\mathbb B(I(X))$ has a decomposition that is a cancellation of the table sequence $(\mathbb B(M_k(X,z)):k\in\mathbb Z)$.
\end{Prop}

\begin{proof}
For convenience write
$$
K^l_{i,j,k}=(E^l_{k,i-k})_{i+j}\quad\text{and}\quad K^\infty_{i,j,k}=(E^\infty_{k,i-k})_{i+j}
$$
for the spectral sequence $E$ in \Cref{spectral}. Now the differentials of $E$ induce the complexes
$$
\xymatrix{
\cdots \ar[r] & K^l_{i+1,j-1,k+l} \ar[r] & K^l_{i,j,k} \ar[r] & K^l_{i-1,j+1,k-l} \ar[r] & \cdots
}
$$
such that
\begin{enumerate}
    \item the cohomology group at $K^l_{i,j,k}$ is $K^{l+1}_{i,j,k}$,
    \item the $K_{i,j,k}^l$ stabilize to $K_{i,j,k}^\infty$ as $l$ increases, 
    \item $\dim_\Bbbk K_{i,j,k}^1=\beta_{i,j}(M_k(X,z))$, and
    \item $\sum_{k\in\mathbb Z}\dim_\Bbbk K_{i,j,k}^\infty=\beta_{i,j}(I(X))$.
\end{enumerate}
For the tables 
$$
B^l_k=(\dim_\Bbbk K^l_{i,j,k}:(i,j)\in\mathbb Z\times\mathbb Z)\quad\text{and}\quad B^\infty_k=(\dim_\Bbbk K^\infty_{i,j,k}:(i,j)\in\mathbb Z\times\mathbb Z)
$$ 
consider the table sequences 
$$
\mathcal B^l=(B^l_k:k\in\mathbb Z)\quad\text{and}\quad\mathcal B^\infty=(B^\infty_k:k\in\mathbb Z).
$$ 
By the above properties of the $K_{i,j,k}^l$ the table sequence
$$
\mathcal B^1=(\mathbb B(M_k(X,z)):k\in\mathbb Z)
$$ 
is a cancellation of a decomposition $\mathcal B^\infty$ of $\mathbb B(I(X))$.
\end{proof}

For the cancellation principle the following can be regarded as the first nontrivial example.

\begin{Ex}
Let $E\subset\mathbb P^3$ be an elliptic normal curve of degree $4$, and take a (general) point $z\in E$. Then
\begin{enumerate}
\item $E_z\subset\mathbb P^2$ is a cubic plane curve, and

\item $K_1(E,z)$ is the homogeneous ideal of a point in $\mathbb P^2$ with $K_1(E,z)\supsetneq K_0(E,z)$.
\end{enumerate}
We have
$$
\mathbb B(M_0(E,z))=
\begin{array}{c|cc}
j\backslash i	&0	&1 \\ \hline
2	&	&	\\
3	&1	&
\end{array}
,\quad\mathbb B(M_1(E,z))=
\begin{array}{c|cc}
j\backslash i	&0	&1 \\ \hline
2	&2	&1	\\
3	&	&1
\end{array}
,\quad\text{and}\quad\mathbb B(M_k(E,z))=0
$$
for all $k\geq 2$. However, any two linearly independent quadratic forms in $I(E)$ do not have nontrivial relations so that the cancellation in \Cref{cancel} must occur between $\beta_{1,2}(M_1(E,z))$ and $\beta_{0,3}(M_0(E,z))$. Therefore, we have shown that the decomposition of $\mathbb B(I(E))$ induced by the cancellation principle consists of $\mathbb B(I(E))$ itself and zero tables. Note that $E\subset\mathbb P^3$ is a complete intersection of two quadrics.
$$
\mathbb B(I(E))=
\begin{array}{c|cc}
j\backslash i	&0	&1 \\ \hline
2	&2	&	\\
3	&	&1
\end{array}
$$
\end{Ex}

Let $M$ be a finitely generated graded $S$-module. Recall that the \emph{regularity} of $M$ relates to the height of the Betti table $\mathbb B(M)=(\beta_{i,j}(M):(i,j)\in\mathbb Z\times\mathbb Z)$ when the rows (resp.\ columns) are indexed by $j$ (resp.\ $i$), that is,
$$
\reg M=\max\{j\in\mathbb Z:\beta_{i,j}(M)\neq 0\textup{ for some }i\in\mathbb Z\}.
$$
But the \emph{projective dimension} of $M$, denoted by $\pd M$, is known to measure the width of $\mathbb B(M)$, which means that
$$
\pd M=\max\{i\in\mathbb Z:\beta_{i,j}(M)\neq0\textup{ for some }j\in\mathbb Z\}.
$$
This satisfies the \emph{Auslander--Buchsbaum formula}
$$
\depth M+\pd M=\depth S,
$$
where $\depth M$ is the \emph{depth} of the irrelevant ideal on $M$.
\begin{Def}
One says that $M_i(X,z)$ \emph{dominates the others}
\begin{enumerate} 
\item \emph{by the regularity} if
\begin{enumerate}
\item $\reg M_i(X,z)\geq\reg M_j(X,z)$ for every $j<i$, and if

\item $\reg M_i(X,z)\geq\reg M_j(X,z)+2$ for every $j>i$, and
\end{enumerate}

\item \emph{by the projective dimension} if
\begin{enumerate}
\item $\pd M_i(X,z)\geq\pd M_j(X,z)+2$ for every $j<i$, and if

\item $\pd M_i(X,z)\geq\pd M_j(X,z)$ for every $j>i$.
\end{enumerate}
\end{enumerate}
\end{Def}

With these conditions the cancellation principle directly implies the result below.

\begin{Prop}\label{dominate}
Let $i\geq0$ be an integer.
\begin{enumerate}
\item If $M_i(X,z)$ dominates the others by the regularity, then we have
$$
\reg X=\reg M_i(X,z).
$$

\item If $M_i(X,z)$ dominates the others by the projective dimension, then we have 
$$
\pd S(X)=\pd M_i(X,z)+1.
$$
\end{enumerate}
\end{Prop}

The following tells us about invariants of an unprojection provided that the associated projection is simple.

\begin{Cor}\label{compu}
Let $\lambda$ be a fake linear form on $Y\subseteq\mathbb P^{r-1}$, and suppose that $\pi_z|_X$ is simple for the unprojection $(X,z)=\Unpr(Y,\lambda)$. Fix a fractional expression $\lambda=a/f\in\Bbbk(\widehat{Y})$ with $a,f$ homogeneous, and put $T=I(Y)+(a,f)\subset R$.
\begin{enumerate}
\item\label{compu1} For the projective scheme $W=V(T)\subseteq\mathbb P^{r-1}$ if $\dim W=\dim Y-1$, then
$$
\deg X=(\deg f+1)\deg Y-
\begin{cases}
\deg W & \textup{if }\dim W=\dim Y-1 \\
0 & \textup{if }\dim W<\dim Y-1.
\end{cases}
$$

\item\label{compu2} If $\reg T>\reg Y+\deg f$, then
$$
\reg X=\reg T-\deg f+1.
$$

\item\label{compu3} If $\pd R/T>\pd S(Y)+2$, then
$$
\pd S(X)=\pd R/T-1.
$$
\end{enumerate}
\end{Cor}

\begin{proof}
We begin with the following diagram with exact rows and column. For the column refer to the formula \eqref{colon}.
$$
\xymatrix{
&&& 0 \ar[d] \\
0 \ar[r] & M_1(X,z) \ar[r] & S(Y)(-1) \ar[r] & \dfrac{R}{K_1(X,z)}(-1) \ar[r] \ar[d]^--a & 0 \\
0 \ar[r] & S(Y) \ar[r]^--f & S(Y)(\deg f) \ar[r] & \dfrac{S(Y)}{(f)}(\deg f) \ar[r] \ar[d] & 0 \\
&&& \dfrac{R}{T}(\deg f) \ar[d] \\
&&& 0
}
$$

\eqref{compu1} It is by the degree formula (\ref{degformula}). Note that $K_1(X,z)$ is equal to $K_\infty(X,z)$ with $z\in X$, and so the column above computes the degree of $\mathbb PC_zX$. \eqref{compu2} Tensoring the diagram above with the residue field $\Bbbk$ of $R$ we see that under the assumption $M_1(X,z)$ dominates the others by the regularity. Apply \Cref{dominate}. \eqref{compu3} Similarly in this case $M_1(X,z)$ dominates the others by the projective dimension.
\end{proof}

\begin{Rmk}
Notice that if $Y\subset\mathbb P^{r-1}$ is a hypersurface, then the homogeneous ideal $T\subset R$ above becomes an ideal with three generators, namely a \emph{three-generated ideal}. Such ideals have a wide range of complexity as shown in \cite{bruns1976jede}. This is a substantial basis for the construction of our counterexamples to the regularity conjecture.
\end{Rmk}

A similar thing holds for a \emph{general} fake linear form of degree $1$.

\begin{Cor}\label{generalcompu}
Suppose that $\depth S(Y)\geq 2$. Take a regular sequence $f,a$ of homogeneous forms on $S(Y)$ with $\deg f=\deg a-1\geq 1$, and let $\lambda=a/f\in\Bbbk(\widehat{Y})$ be the induced fake linear form of degree $1$ on $Y\subseteq\mathbb P^{r-1}$. Then for the unprojection $(X,z)=\Unpr(Y,\lambda)$ the projection $\pi_z|_X$ is simple, and $X\subset\mathbb P^r$ has invariants
$$
\deg X=(\deg f+1)\deg Y,\quad\reg X=\reg Y+\deg f,\quad\text{and}\quad\pd S(X)=\pd S(Y)+1.
$$
\end{Cor}

\begin{proof}
The formula \eqref{colon} says that
$$
\frac{K_i(X,z)}{K_0(X,z)}=(f(a,f)^{i-1}:a^i)\subseteq((f):a^i)=(f)
$$
for every $i\geq 1$ with equality when $i=1$. Thus, it follows that $\pi_z|_X:X\setminus z\to Y$ is simple together with
$$
M_1(X,z)\cong S(Y)(-\deg f-1).
$$
By \Cref{compu} we have $\deg X=(\deg f+1)\deg Y$. Since $\reg M_1(X,z)=\reg S(Y)(-\deg f-1)=\reg S(Y)+\deg f+1=\reg Y+\deg f\geq\reg M_0(X,z)$, \Cref{dominate} tells us that $\reg X=\reg M_1(X,z)=\reg Y+\deg f$. For the projective dimension part, we use the cancellation principle. Note that $\pd M_1(X,z)=\pd S(Y)=\pd I(Y)+1=\pd M_0(X,z)+1$. Put $p=\pd M_1(X,z)$ and $q=\max\{j\in\mathbb Z:\beta_{p,j}(M_1(X,z))\neq0\}$ so that $\beta_{p,q}(M_1(X,z))>0$. We however obtain $\pd M_0(X,z)=\pd S(Y)-1=\pd M_1(X,z)-1=p-1$ and $\beta_{p-1,q+1}(M_0(X,z))=\beta_{p,q}(S(Y))=\beta_{p,q+\deg f+1}(M_1(X,z))=0$. We are done.
\end{proof}

\bigskip

\section{Families of homogeneous ideals with high regularity}\label{secfamily}

In this section we focus on sequences of homogeneous ideals whose values of the regularity grow fast.

\begin{Ex}\label{Caviglia}
In \cite[Section 4.2]{caviglia2004koszul} Caviglia considered homogeneous ideals
$$
I_d=(z_2^{d-1}y_1-y_2^{d-1}z_1,y_1^d,z_1^d)\subset\Bbbk[y_1,z_1,y_2,z_2].
$$
Finding its Gr{\"o}bner basis with respect to the reverse lexicographic order he showed that
\begin{equation}\label{reg2}
\reg I_d=d^2-1   
\end{equation}
for every $d\geq 2$. Furthermore, for the homogeneous ideals
$$
I_{k,d}=(z_{i+1}^{d-1}y_i-y_{i+1}^{d-1}z_i:i=1,\ldots,k-1)+(y_1^d,z_1^d)\subset\Bbbk[y_1,z_1,\ldots,y_k,z_k]
$$
he expected with computational evidence that for each $k\geq 3$ there exists some univariate (rational) polynomial $p_k$ of order $k$ such that
$$
\reg I_{k,d}=p_k(d).
$$
Notice that in \cite[Remark 2.7]{borna2015some} Borna and Mohajer made the same claim for the homogeneous ideals 
$$
J_{k,d}=(z_{i+1}^{d-1}y_i-y_{i+1}^{d-1}z_i,y_i^d,z_i^d:i=1,\ldots,k-1)\subset\Bbbk[y_1,z_1,\ldots,y_k,z_k].
$$
They also observed that
\begin{equation}\label{reg3}
\reg J_{3,d}=d^3-3d^2+5d-3   
\end{equation}
and presumed that
\begin{equation}\label{reg4}
\reg J_{4,d}=d^4-4d^3+5d^2+d-3.  
\end{equation}
\end{Ex}

\begin{Ex}\label{BMNSSS}
In \cite{beder2011ideals} Beder, McCullough, N{\'u}{\~n}ez-Betancourt, Seceleanu, Snapp, and Stone constructed interesting three-generated homogeneous ideals denoted by $I_{2,(2,\ldots,2,1,d)}$ and raised the speculation that for the $k$-tuple $(2,\ldots,2,1,d)$ the estimate
$$
\reg I_{2,(2,\ldots,2,1,d)}=\Omega(d^k)
$$
holds. With minor modifications these ideals are formed as follows.

Let $k\geq2$ and $d\geq 0$ be integers. We first set
$$
F_{k-1}=y_{k-1}^{d+2},\quad G_{k-1}=z_k^{d+1}y_{k-1}-y_k^{d+1}z_{k-1}\quad\text{and}\quad H_{k-1}=z_{k-1}^{d+2},
$$
and then inductively put
$$
F_i=y_i^{d+2k-2i},\quad G_i=F_{i+1}y_i^2+G_{i+1}y_iz_i+H_{i+1}z_i^2\quad\text{and}\quad H_i=z_i^{d+2k-2i}
$$
for each $i=k-2,k-3,\ldots,1$. Now the homogeneous ideal $I_{2,(2,\ldots,2,1,d)}$ is defined by
$$
I_{2,(2,\ldots,2,1,d)}=(F_1,G_1,H_1)\subset\Bbbk[y_1,z_1,\ldots,y_k,z_k].
$$ 
\end{Ex}

The following partially verifies Borna--Mohajer's assertion in \Cref{Caviglia} and is a foundation for our counterexamples. Its proof is inspired by that of \cite[Proposition 7.1]{caviglia2019regularity}.

\begin{Thm}\label{family}
Fix an integer $k\geq 2$, and let $b_1,c_1,a_2,\ldots,b_{k-1},c_{k-1},a_k\geq 1$ be integers satisfying $1\leq a_i<\min\{b_i,c_i\}$ for all $i=2,\ldots,k-1$. Then the homogeneous ideal
$$
I=(z_{i+1}^{a_{i+1}}y_i-y_{i+1}^{a_{i+1}}z_i,y_i^{b_i},z_i^{c_i}:i=1,\ldots,k-1)\subset Q=\Bbbk[y_1,z_1,\ldots,y_k,z_k]
$$
has invariants
\begin{equation}\label{regk}
\reg I\geq (\min\{b_1,c_1\}-1)a_2\cdots a_k-\sum_{i=3}^ka_i\cdots a_k+\sum_{i=2}^ka_i+\max\{b_1,c_1\}-1
\end{equation}
and $\pd Q/I=2k$.
\end{Thm}

\begin{proof}
Put $A=\Bbbk[y_k,z_k]\subset Q$, write $a_1=\max\{b_1,c_1\}$, and let $M\subseteq Q/I$ be the $A$-submodule generated by all the monomials
$$
m(d_1,\ldots,d_{k-1})=\prod_{i=1}^{k-1}y_i^{a_i-d_i}z_i^{d_i},
$$
where $0\leq d_i\leq a_i$ for all $1\leq i\leq k-1$. It is an $A$-direct summand of $Q/I$ since the monomials have the same degree in $y_i$ and $z_i$ for each $1\leq i\leq k-1$. Note that $Q/I$ is finitely generated as an $A$-module, which implies that the $Q$- and $A$-module structures on $Q/I$ give the same local cohomology groups. Due to vanishing of local cohomology groups (see \cite[Corollary 4.5 and Proposition A1.16]{eisenbud2005geometry}) we have
$$
\reg I=\reg Q/I+1\geq\reg M+1\quad\text{and}\quad\pd Q/I\geq\pd M+2k-2.
$$
Thanks to these inequalities it would be enough to compute $\reg M$ and $\pd M$.

Let us find the minimal free resolution of $M$. For integers $0\leq d\leq\alpha$, where
$$
\alpha=a_1\cdots a_{k-1}+a_2\cdots a_{k-1}+\cdots+a_{k-1},
$$
we set
$$
m_d=m(d_1,\ldots,d_{k-1})
$$
when 
$$
d=d_1a_2\cdots a_{k-1}+d_2a_3\cdots a_{k-1}+\cdots+d_{k-1}
$$
with integers $0\leq d_i\leq a_i$ for all $1\leq i\leq k-1$. They are well-defined by the ``carry-borrow" rule
$$
m(\ldots,d_i,d_{i+1},\ldots)=m(\ldots,d_i+1,0,\ldots)
$$
for the case where $d_{i+1}=a_{i+1}$ and $d_i<a_i$, which is induced by the generators $z_{i+1}^{a_{i+1}}y_i-y_{i+1}^{a_{i+1}}z_i\in I$ with $1\leq i\leq k-2$. For a similar reason we find that
\begin{equation}\label{rel}
m_dz_k^{a_k}=m_{d+1}y_k^{a_k}    
\end{equation}
for any $0\leq d<\alpha$. Notice that if either $d\leq\alpha-\beta$ or $\gamma\leq d$ holds with
$$
\beta=b_1a_2\cdots a_{k-1}\quad\textup{and}\quad\gamma=c_1a_2\cdots a_{k-1},
$$
then $m_d=0$ in $M$, for $z_1^{c_1}$ and $y_1^{b_1}$ are members of $I$. We claim that $m_{\alpha-\beta+1},\ldots,m_{\gamma-1}\in M$ are nonzero generators whose $A$-linear relations
$$
f=\sum_{d=\alpha-\beta+1}^{\gamma-1}f_dm_d\in I
$$
are spanned by the relations \eqref{rel}.

Indeed, comparing the degrees of $f$ and the generators of $I$ in $y_i$ and $z_i$ for each $2\leq i\leq k-1$ one sees that 
$$
f\in(z_{i+1}^{a_{i+1}}y_i-y_{i+1}^{a_{i+1}}z_i:1\leq i\leq k-1)+(y_1^{b_1},z_1^{c_1}).
$$
Plug
$$
y_i=y_{k-1}^{a_{i+1}\cdots a_{k-1}}\quad\text{and}\quad z_i=z_{k-1}^{a_{i+1}\cdots a_{k-1}}
$$
into the $m_d$ and $f$ for all $1\leq i\leq k-2$ so that we reach the images
$$
\overline{m_d}=y_{k-1}^{\alpha-d}z_{k-1}^d\quad\textup{and}\quad\overline{f}=\sum_df_dy_{k-1}^{\alpha-d}z_{k-1}^d\in (z_k^{a_k}y_{k-1}-y_k^{a_k}z_{k-1},y_{k-1}^\beta,z_{k-1}^\gamma)
$$
in $A[y_{k-1},z_{k-1}]$. Our claim has been shown.

Consequently, by letting $N=\beta+\gamma-\alpha-1$ the minimal free resolution of $M$ has differentials
$$
\begin{pmatrix}
-z_k^{a_k} 	& y_k^{a_k} 	& 			& 			& 				\\
			& -z_k^{a_k}	& y_k^{a_k}& 			& 				\\
			& 				& \ddots	& \ddots 	& 				\\
			& 				& 			& -z_k^{a_k}& y_k^{a_k} 			
\end{pmatrix}
\quad\text{and}\quad
\begin{pmatrix}
(y_k^{a_k})^N \\
(y_k^{a_k})^{N-1}z_k^{a_k} \\
(y_k^{a_k})^{N-2}(z_k^{a_k})^2 \\
\vdots \\
(z_k^{a_k})^N
\end{pmatrix}
,
$$
and the Betti table of $M$ looks like
$$
\begin{array}{c|ccc}
j\backslash i 	&0			&1			&2			\\ \hline
j_0				&N		&			&			\\
j_1				&			&N+1	&			\\
j_2				&			&			&1
\end{array}
$$
in the setting of
\begin{enumerate}
\item $j_0=\sum_{i=1}^{k-1}a_i$,

\item $j_1=j_0+a_k-1=\sum_{i=1}^ka_i-1$, and 

\item $j_2=j_1+a_kN-1=(\beta+\gamma-\alpha-1)a_k+\sum_{i=1}^ka_i-2$.
\end{enumerate}
We are done.
\end{proof}

\begin{Rmk}
If 
$$
(b_1,c_1,a_2,\ldots,b_{k-1},c_{k-1},a_k)=(d,d,d-1,\ldots,d,d,d-1)
$$
for an integer $d\geq 2$, then the inequality \eqref{regk} reads
$$
\reg I\geq(d-1)^k-\sum_{i=1}^{k-2}(d-1)^i+k(d-1).
$$
Therefore, we have verified the lower bound versions of \eqref{reg2}, \eqref{reg3}, and \eqref{reg4}, which becomes evidence that the inequality \eqref{regk} is sharp.
\end{Rmk}

Generalizing a technique of Caviglia, Chardin, McCullough, Peeva, and Varbaro we establish sequences of three-generated homogeneous ideals with the same growth rates of regularity as those from \Cref{family}.

\begin{Prop}[{cf.\ \cite[Proposition 7.1]{caviglia2019regularity}}]\label{fixing}
Let $Q\subset R$ be a polynomial ring obtained by eliminating two variables in $R$, say $y_0$ and $z_0$, and take a homogeneous ideal $U=(g_0,\ldots,g_s)\subset Q$ and positive integers $\Delta$ and $\delta$. Suppose that
\begin{enumerate}
\item $g_s$ is nonzero,

\item $\deg g_i=mi+\deg g_0$ for a common integer $m\geq 0$, and

\item the inequalities $\Delta>s$ and $\delta>(m+1)s$ hold.
\end{enumerate}
We put
$$
G=g_sy_0^s+g_{s-1}y_0^{s-1}z_0^{m+1}+\cdots+g_0z_0^{(m+1)s}\quad\text{and}\quad T=(G,y_0^\Delta,z_0^\delta).
$$
Then we have
$$
\deg R/T\leq\delta s,\quad\reg T\geq\reg U+\Delta+\delta-2,\quad\text{and}\quad\pd R/T\geq\pd Q/U+2.
$$
Moreover, the following hold.
\begin{enumerate}
    \item If $(m+1)(\Delta-s+1)\geq\delta$, then $\deg R/T=\delta s$.
    \item If $g_s$ is irreducible, then so is $G$ provided that the $g_i$ are linearly independent over $\Bbbk$.
\end{enumerate}
\end{Prop}

\begin{proof}
Mimicking the proof of \cite[Proposition 7.1]{caviglia2019regularity} with the monomial $y^{\Delta-1}z^{\delta-1}$ we have the desired lower bounds of $\reg T$ and $\pd R/T$. Also, since the radical of $T$ is $(y_0,z_0)$, we obtain
$$
\deg R/T=\dim_KR/T\otimes_QK,
$$
where $K$ is the fraction field of $Q$. Let $V$ be the $K$-vector space $R/T\otimes_QK$. By dividing elements of $V$ by $G$ the vector space $V$ is is shown to be generated by the monomials
\begin{equation}\label{monomials}
\begin{array}{ccc}
y_0^0z_0^{\delta-1},	&\ldots,	&y_0^{s-1}z_0^{\delta-1},	\\
\vdots					&			&\vdots					\\
y_0^0z_0^0,			&\ldots,	&y_0^{s-1}z_0^0.
\end{array}
\end{equation}
Hence, the upper bound $\deg R/T\leq\delta s$ holds.

Next assume that $(m+1)(\Delta-s+1)\geq\delta$. Introduce a new grading $\deg'$ having values $\deg' y_0=m+1$, $\deg' z_0=1$, and $\deg'=0$ on $Q$. According to the values $j$ of this grading we get the $K$-vector space decomposition
$$
V=\bigoplus_{j\in\mathbb Z}V_j.
$$
We divide into two cases: $j<\Delta(m+1)$ and $j\geq\Delta(m+1)$. If $j<\Delta(m+1)$, and if we express $j$ as $j=a(m+1)+b$ with integers $0\leq b<m+1$ and $0\leq a<\Delta$, then the summand $V_j$ is generated by the set $\{y_0^az_0^b,y_0^{a-1}z_0^{b+m+1},\ldots\}$ and becomes the cokernel of a matrix of the form
$$
\begin{pmatrix}
g_s 		&		&			\\
g_{s-1} 	&\ddots&			\\
\vdots 		&\ddots&g_s 		\\
			&		&g_{s-1} 	\\
			&		&\vdots
\end{pmatrix}
.
$$
However, in the range $j\geq\Delta(m+1)$ some matrix of the form
$$
\begin{pmatrix}
\cdots 	&g_{s-1}	&g_s 		&			&			&		\\
		&			&\ddots 	&\ddots 	&			&		\\
		&			&\cdots	&g_{s-1}	&g_s 		&
\end{pmatrix}
$$
works for a presentation of $V_j$. It is due to the assumption $(m+1)(\Delta-s+1)\geq\delta$. So one may observe that the monomials (\ref{monomials}) form a basis of $V$.

Finally, let us see the irreducibility of $G$. Refer to the fact that $G$ is homogeneous with respect to both the standard grading and $\deg'$. Suppose that 
$$
G=(a_ty_0^t+\cdots+a_0z_0^{(m+1)t})(b_uy_0^u+\cdots+b_0z_0^{(m+1)u})
$$ 
for integers $t,u\geq0$ with $t+u=s$ and homogeneous forms $a_i,b_j\in Q$ with $\deg a_i=mi+\deg a_0$ and $\deg b_j=mj+\deg b_0$. We find that $a_tb_u=g_s$, and so we may assume that $a_t$ is a nonzero constant. Since $\deg a_i=mi+\deg a_0$, and since the $g_i$ are $\Bbbk$-linearly independent, we have $t=0$, and thus the first factor above is a nonzero constant, which means that $G$ is irreducible.
\end{proof}

We apply the proposition above to suitable homogeneous ideals in \Cref{family}.

\begin{Cor}\label{Ingredient}
Fix an integer $k\geq 2$, and let $d\geq 2$ be an integer. Then the polynomial ring $R=\Bbbk[y_0,z_0,\ldots,y_k,z_k]$ admits an irreducible $(d+6k-8)$-form $G_{k,d}\in R$ and a three-generated homogeneous ideal 
$$
T_{k,d}=(G_{k,d},y_0^{6k-6},z_0^{6k-7})\subset R
$$ 
such that 
$$
\reg T_{k,d}=\Omega(d^k),\quad\text{and}\quad\pd R/T_{k,d}=2k+2
$$
together with
$$
\deg R/(G_{k,d},y_0^{(6k-6)i},z_0^{(6k-7)i})=(6k-7)(3k-4)i
$$
for all $i\geq 1$.
\end{Cor}

\begin{proof}
Write $Q=\Bbbk[y_1,z_1,\ldots,y_k,z_k]$. By setting
$$
(b_i,c_i,a_{i+1})=(d+3i-3,d+3i-2,d+3i-2),\quad 1\leq i\leq k-1,
$$
\Cref{family} produces a homogeneous ideal $U=(g_0,\ldots,g_{3k-4})\subset Q$ whose generators are arranged as
$$
g_{3i}=y_{i+1}^{d+3i},\quad g_{3i+1}=z_{i+1}^{d+3i+1},\quad\textup{and}\quad g_{3i+2}=y_{i+2}^{d+3i+1}z_{i+1}-z_{i+2}^{d+3i+1}y_{i+1}
$$
for integers $1\leq i\leq k-1$. \Cref{fixing} accomplishes the proof.
\end{proof}

\bigskip

\section{Proof of the main theorem}\label{secproof}

In this section we show the main theorem.

\begin{proof}[\noindent{\bf Proof of \Cref{ce}.}]
Fix an integer $k\geq2$, and set $d\geq 2$. Consider the three-generated homogeneous ideal $$
T_{k,d}=(G_{k,d},y_0^{6k-6},z_0^{6k-7})\subset R=\Bbbk[y_0,z_0,\ldots,y_k,z_k]
$$ 
formed in \Cref{Ingredient}. Take
$$
Y_d=V(G_{k,d})\subset\mathbb P^{2k+1}\quad\text{and}\quad\lambda=\frac{y_0^{6k-6}}{z_0^{6k-7}}\in\Bbbk(\widehat{Y_d})
$$
to construct an unprojection $(X_d,z)=\Unpr(Y_d,\lambda)$ in $\mathbb P^{2k+2}$. We claim that the sequence 
$$
(X_d\subset\mathbb P^{2k+2}:d\geq 2)
$$ 
is the desired one with dimension $n=2k$ and codimension $e=2$.

Indeed, notice that part of \Cref{fixing} induces upper bounds
$$
\deg W_i\leq\deg R/(G_{k,d},y_0^{(6k-6)i},z_0^{(6k-7)i})\leq i\deg W_1
$$
for the subschemes $W_i=V((a,f)^i)\subset Y_d\subset\mathbb P^{2k+1}$. By \Cref{degsimple} and \Cref{compu} we find that $\pi_z|_{X_d}$ is simple and that 
$$
\lim_{d\to\infty}\deg X_d=\infty,\quad\reg X_d=\Omega((\deg X_d)^k),\quad\textup{and}\quad\pd S(X_d)=2k+1
$$
due to invariants of $T_{k,d}$.

For higher codimension cases $e>2$ conduct the following process. Suppose given a counterexample $X_d\subset\mathbb P^{2k+e-1}$ of dimension $2k$ and codimension $e-1\geq 2$, and say that $S(X_d)$ has depth $2$. Then pick a \emph{general} fake linear form of degree $1$ on $X_d\subset\mathbb P^{2k+e-1}$, and apply \Cref{generalcompu}. We have therefore constructed a sequence 
$$
(X_d\subset\mathbb P^{2k+e}:d\geq 2)
$$ 
of nondegenerate projective $2k$-folds satisfying
$$
\lim_{d\to\infty}\deg X_d=\infty,\quad\reg X_d=\Omega((\deg X_d)^k),\quad\text{and}\quad\pd S(X_d)=2k+e-1
$$
for every $e\geq 2$. On the other hand for $(2k-1)$-fold counterexamples consider a general hyperplane section of each $2k$-fold $X_d$. It is valid thanks to $\depth S(X_d)=2$. We are done.
\end{proof}

\begin{Rmk}
It can be said that our counterexamples work as they carry very bad singularities. For instance, let us look at the case $(X,z)=(X_d,z)$ of dimension $2k$ and codimension $2$ in the proof above. Then the point $z\in\mathbb P^{2k+2}$ turns out to be a singular point of $X$, and its projectivized tangent cone $\mathbb PC_zX$ is cut out by $K_1(X,z)$, having exceedingly high regularity compared to the multiplicity of $X$ at $z$. Recall that the pathological syzygies of $I(X)$ immediately come from those of $M_1(X,z)$ in view of the cancellation principle.
\end{Rmk}

Finally, we present in detail a counterexample to the regularity conjecture obtained from a homogeneous ideal in \Cref{BMNSSS} by applying our method. Many of the computations below are done by the computer algebra system \emph{Macaulay2} \cite{M2} over the field $\mathbb Q$ of rational numbers.

\begin{Ex}
Let $R=\Bbbk[y_0,z_0,y_1,z_1,y_2,z_2]$ be the homogeneous coordinate ring of $\mathbb P^5$, and take $Y\subset\mathbb P^5$ to be the hypersurface given by the irreducible homogeneous polynomial
$$
G=y_1^4y_0^2+(z_2^3y_1-y_2^3z_1)y_0z_0+z_1^4z_0^2
$$
that appears in \Cref{BMNSSS}. Notice that the three-generated homogeneous ideal 
$$
T=(G,y_0^4,z_0^3)\subset R
$$ 
satisfies
$$
\deg R/T=6,\quad\reg T=20,\quad\text{and}\quad\pd R/T=6.
$$
Now with the fake linear form 
$$
\lambda=y_0^4/z_0^3\in\Bbbk(\widehat{Y})
$$ 
of degree $1$ on $Y\subset\mathbb P^5$ consider the nondegenerate projective fourfold
$$
(X,z)=\Unpr(Y,\lambda)\quad\text{in}\quad\mathbb P^6.
$$  
Then having
$$
\deg X=18,\quad\reg X=18,\quad\text{and}\quad\pd S(X)=5,
$$
we have established a fourfold counterexample to the regularity conjecture. Here the regularity is exactly one more than the conjectural bound of Eisenbud and Goto. The Betti table $\mathbb B(I(X))$ is computed to be the following.
$$
\begin{array}{c|ccccccccccccccc}
i\backslash j	&4	&5	&6	&7	&8	&9	&10	&11	&12	&13	&14	&15	&16	&17	&18\\ \hline
0      			&1	&	&1	&2	&	&	&3		&		&1		&1		&1		&		&1		&		&1\\
1      			&	&	&	&3	&	&	&6		&		&4		&2		&4		&		&4		&		&4\\
2      			&	&	&	&	&	&	&3		&		&5		&1		&6		&		&6		&		&6\\
3      			&	&	&	&	&	&	&		&		&2		&		&4		&		&4		&		&4\\
4    			&	&	&	&	&	&	&		&		&		&		&1		&		&1		&		&1
\end{array}
$$
Moreover, thanks to $\depth S(X)=2$ a general hyperplane section
$$
X\cap\mathbb P^5\subset\mathbb P^5
$$ 
is a threefold counterexample with the same invariants as above. Compare it with the projective threefold in \cite[Example 4.7]{mccullough2018counterexamples}.

Let us describe singularities of the fourfold $X$. The singular locus of $X$ is
$$
\Sing X=C\cup\Lambda_1\cup\cdots\cup\Lambda_4\cup\mathbb P^3\subset\mathbb P^6,
$$
where $C\subset\mathbb P^6$ is a quartic curve in a $2$-plane, the $\Lambda_i\subset\mathbb P^6$ are $2$-planes in a certain order, and $\mathbb P^3\subset\mathbb P^6$ is a $3$-plane. Here the location of $z\in\mathbb P^6$ is obtained as
$$
z\in(C\cap\Lambda_1\cap\Lambda_2)\setminus(\Lambda_3\cup\Lambda_4\cup\mathbb P^3)\subset\Sing X.
$$
The projectivized tangent cone $\mathbb PC_zX\subset\mathbb P^5$ is of degree $12$, its homogeneous ideal has regularity $17$, and its support is a $3$-plane in $\mathbb P^5$.
\end{Ex}

\bigskip

\noindent{\bf Acknowledgements.} 
We would like to thank Sijong Kwak for his advice. We are also grateful to Robert Lazarsfeld for his interest. Experiments with Macaulay2 \cite{M2} were very helpful for this study. This work was supported by the National Research Foundation of Korea (NRF) grant funded by the Korea government (No.\ 2019R1A2C3010487).

\bigskip

\nocite{*}
\bibliographystyle{alpha}
\bibliography{Unprojection}


\end{document}